\documentclass{amsart}
\usepackage{graphicx} 
\usepackage{amsrefs}
\usepackage[utf8]{inputenc}
\usepackage{amsthm,enumerate}
\usepackage{amsmath}
\usepackage{mathrsfs}
\usepackage{amssymb}
\usepackage{enumitem}
\usepackage{mathtools}
\usepackage{hyperref}
\usepackage[top=1in, bottom=1in, right=1in, left=1in]{geometry}
\usepackage[capitalise]{cleveref}
\usepackage[dvipsnames]{xcolor}
\usepackage[UKenglish]{babel}
\usepackage{tikz-cd}
\usepackage{thmtools}

\theoremstyle{plain}
\newtheorem{thm}{Theorem}[section]
\newtheorem{cor}[thm]{Corollary}
\newtheorem{lemma}[thm]{Lemma}

\newtheorem{prop}[thm]{Proposition}
\newtheorem{question}[thm]{Question}

\theoremstyle{definition}
\newtheorem{defn}[thm]{Definition}

\newenvironment{theorem_no_number}[1][]{\begin{trivlist}
\item[\hskip \labelsep {\bfseries Theorem \def\temp{#1}\ifx\temp\empty  #1\else  #1\fi
.}] \itshape}  {\end{trivlist}}

\newcommand{\R}{\mathbb{R}}
\newcommand{\N}{\mathbb{N}}
\newcommand{\Q}{\mathbb{Q}}
\newcommand{\Z}{\mathbb{Z}}

\newcommand{\im}{\operatorname{im}}
\newcommand{\dom}{\operatorname{dom}}
\newcommand{\Sym}{\operatorname{Sym}}

\newcommand{\cw}{\operatorname{cmw}}
\newcommand{\id}{\operatorname{id}}
\newcommand{\makeset}[2]{\left\lbrace #1 \;\middle|\;
 \begin{tabular}{@{}l@{}}
   #2
  \end{tabular}
  \right\rbrace}

\usepackage{hyperref}
\usepackage{cleveref}
\title{Commutative decomposition of infinite symmetric groups and transformation monoids}
\author{Luna Elliott and Alex Levine}
\begin{document}
\maketitle
\begin{abstract}
The commutative subgroup width of a group \(G\) is the smallest \(k\) such
that there are abelian subgroups \(A_0,A_1,\ldots,A_{k-1}\leq G\) with
\(G=A_0A_1\ldots A_{k-1}\). Commutative (inverse) submonoid width is defined analogously.
In 2002, Ab\'{e}rt showed, rather surprisingly, that the commutative
subgroup width of the symmetric group on an infinite set is always finite. It
was later shown by Seress that it is always bounded above by \(14\). We answer a
question of Seress and show that in fact the commutative subgroup width of \(\Sym(\N)\) is
at most \(9\). 
We improve the best known lower bound to \(4\). We also study
standard monoid analogues of the symmetric group; showing that the commutative
submonoid widths of the full transformation monoid \(\N^\N\), the partial
transformation monoid \(P_\N\) and the symmetric inverse monoid \(I_\N\) are
exactly \(3\). We conclude by showing that the commutative inverse submonoid
width of any infinite symmetric inverse monoid is always infinite.

%
\end{abstract}
\section{Introduction}

A group \(G\) is said to have \emph{finite commutative subgroup width} if
there exists \(n\in \N\) such that \(G=C_1C_2\cdots C_n\), where
\(C_1,\ldots,C_n\) are commutative subgroups of \(G\). We write
\(\operatorname{cgw}(G)\) for the least such \(n\) if it exists, otherwise we
write \(\operatorname{cgw}(G)=\infty\). Commutative inverse submonoid width and
commutative submonoid width are defined analogously and are denoted by
\(\operatorname{ciw}\) and \(\operatorname{cmw}\) respectively. 
As we will see
later, these all coincide in the case of a group, although
\(\operatorname{ciw}(M)\) and \(\operatorname{cmw}(M)\) may differ for an
inverse monoid \(M\).

A related notion is that of \emph{bounded generation}, also called \emph{finite
cyclic width}, which is the same definition as having finite commutative
subgroup width, except requiring that all of the relevant subgroups are cyclic,
rather than just commutative. Naturally, bounded generation implies finite
commutative subgroup width. There exist various studies of bounded generation,
thus giving (finite) upper bounds for the commutative subgroup width in a number
of examples, including \(\operatorname{SL}_n(\Z)\) for \(n \geq 3\)
\cites{MiklosLubotzkyPyber, CarterKeller, Loukanidis, NikolovSury,
MorganRapinchukSury, ErovenkoRapinchuk, Morrisbutnot}.

Ab\'{e}rt directly studied commutative subgroup width in symmetric groups,
showing that if \(X\) is an infinite set, then the symmetric group \(\Sym(X)\)
on \(X\) has finite commutative subgroup width \cite{abert2002symmetric} (of at most \(289\)). 
Seress drastically improved this upper bound to
\(14\) \cite{abelianprod}. Ito also proved that \(3\) is a lower bound for any
non-metabelian group \cite{Ito}, which of course includes \(\Sym(X)\), as Seress
also noted. Ab\'{e}rt also considers the finite symmetric groups, providing
upper and lower bounds in each case (the latter of which are generally much
sharper than \(3\)).

Our purpose is to study the `standard' analogues of the infinite symmetric group
that occur in group and semigroup theory, including the symmetric group itself.
We begin with the symmetric group on the naturals. We sharpen the upper
bound to \(9\) and the lower bound to \(4\). This proves a conjecture of 
Seress \cite{abelianprod}.

\begin{theorem_no_number}[A]
   \(4 \leq \operatorname{cgw}(\Sym(\N)) \leq 9\). 
\end{theorem_no_number}

Additionally, we consider the three analogues of the symmetric group in
semigroup theory: the full transformation monoid \(\N^\N\), the partial
transformation monoid \(P_\N\) and the symmetric inverse monoid \(I_\N\). In
each case, we show that commutative submonoid width is exactly \(3\), contrasting
with the symmetric group.

\begin{theorem_no_number}[B]
    \(\operatorname{cmw}(\N^\N) = \operatorname{cmw}(P_\N) = \operatorname{cmw}(I_\N) = 3\).
\end{theorem_no_number}

Unlike the case of groups where commutative subgroup and submonoid widths coincide, this
is not the case for inverse monoids. Indeed, we show that \(I_\N\) has infinite
commutative inverse submonoid width, in stark contrast with its commutative submonoid width
of \(3\).

\begin{theorem_no_number}[C]
    \(\operatorname{ciw}(I_\N) = \infty\).
\end{theorem_no_number}

In order to prove these results, we use some topological algebra; in particular,
we apply the Baire Category Theorem to these monoids when viewed as topological
monoids. We aim to make this accessible to non-topologists, and all topology
used is fully contained within Section~\ref{sec:top}. The use of the Baire
Category Theorem to show algebraic facts in this manner is not new: it has been
used previously to show that every element of \(\Sym(\N)\) (and certain other
examples) is a commutator, see for example \cite[Proposition
4.2.12]{macpherson}. 
 The upper bounds do not require this topological argument, and the upper bounds for \(X^X\), \(I_X\), and \(P_X\) can be applied to any infinite set \(X\). Our argument for the upper bound on \(\Sym(X)\) requires only that \(|X|\) is an infinite regular cardinal.
An issue with applying our topological methods for
higher cardinals is that the usual version of the Baire Category Theorem only
applies to completely metrisable topological spaces, and (at least using the
analogue of the topologies we have) only \(\Sym(\N)\), \(\N^\N\), \(P_\N\) and
\(I_\N\) have this property. There do exist analogues of the Baire Category
Theorem for higher cardinalities, however even if we resolved this issue, our
arguments use the fact that if \(\kappa\) is a non-zero cardinal strictly below
\(\aleph_0\), then \(\aleph_0^{\kappa}=\aleph_0\). If the continuum hypothesis
is false, then even \(\aleph_1\) will fail to satisfy this condition.

Section~\ref{sec:top} covers the proof of \cref{prop:topfact}, which is the only topological proof we include. Section~\ref{sec:symgroup} covers the proof of Theorem A, and includes a corollary about the commutative submonoid width of the monoid of injective transformations
of \(\N\). Section~\ref{sec:monoids} covers the proof of Theorem B.
\section{Some Topology}\label{sec:top}
We require a proposition that can be proved easily using topological methods. We
provide a short proof of this, but since the proof is topological, we will first
recall the prerequisite concepts, before we give the proof. For sets \(X,
Y\subseteq \N\), we denote the set of bijections from \(X\) to \(Y\) by \(
\operatorname{Bijections}(X,Y)\).   
\begin{prop}\label{prop:topfact}
    Let \(X\) and \(Y\) be infinite subsets of \(\N\). Suppose \((S_i)_{i \in \N}\) 
    is a sequence of subsets of \(\operatorname{Bijections}(X, Y)\) satisfying
    \[
        \bigcup_{i\in \N} S_i=\operatorname{Bijections}(X,Y).
    \]
    Then there exists \(i\in \N\) and \(p\in I_{\N}\) from
    \(X\) to \(Y\) such that:
    \begin{itemize}
        \item \(\dom(p)\) is finite;
        \item \(S_i\) is uncountable; 
        \item For all \(p' \in I_\N\) with finite domain extending \(p\), there
        exists \(s\in S_i\) with \(s|_{\dom(p')} =p'\). 
    \end{itemize}
\end{prop}

The first notion we will need before we can prove \cref{prop:topfact} is the definition of the topology
we will be using. In this case, we need only define the topology on \(P_\N\).
Recall first that \(P_{\N}\) denotes the set of functions from subsets of \(\N\) to \(\N\) and if \(f,g\in P_\N\) then \(fg\) is the function \(x\mapsto ((x)f)g\) with domain \(\dom(f)\cap (\dom(g))f^{-1}\).
The monoid \(I_\N\leq P_\N\) consists of the bijections between subsets of \(\N\), and \(\N^\N\) is the monoid of elements of \(P_\N\) with domain \(\N\).
\begin{defn}\label{def:topologyonpartials}
     We consider \(P_{\N}\) as a monoid with the topology with basis comprising the sets: 
     \[U_{p, X} \coloneqq \makeset{f\in P_{\N}}{\(f|_{X} =p\)},\]
     where \(p\in P_{\N}\) and \(X\subseteq \N\) is a finite set containing \(\dom(p)\).
    For infinite sets \(X, Y\subseteq \N\), we consider \( \textrm{Bijections}(X,Y)\) to be a topological space with the subspace topology inherited from \(P_{\N}\).    
\end{defn}
From \cite{Lunaetal}*{Theorem 5.10}, the topology on \(P_{\N}\) is completely metrisable.
It is well-known (see for example \cite{KechrisBook}*{Theorem 3.11}) that a
subspace of a completely metrisable space is completely metrisable if and only
if it can be expressed as a countable intersection of open sets. Note that the complete
metric used here may not be the `natural' one inherited from the superspace.
We have that for any finite \(D \subseteq X\), the set
\[
    \makeset{f\in P_{\N}}{\(f|_{D}\) is an injective map from \(D\cap X\) to \(Y\)},
\]
is open. Moreover, the intersection of all of these sets is the set of injective
maps from \(X\) to \(Y\). Similarly, each of the sets
\(\makeset{f\in P_{\N}}{\(y\in \im(f)\)}\)
for \(y\in Y\) are open. The intersection of all of these sets is precisely \(
\textrm{Bijections}(X,Y)\). In particular, the set \( \textrm{Bijections}(X,Y)\)
is always completely metrisable with the subspace topology from \(P_\N\). From the case \(X=Y=\N\), we have that
\(\Sym(\N)\) is completely metrisable. We can thus apply the Baire Category
Theorem to these spaces. To do this we require the definition of meagre.

\begin{defn}
    A subset \(S\) of a topological space \(X\) is called \emph{nowhere dense} if \(\bar{S}\) contains
    no non-empty open set. A subset \(M\) is \emph{meagre} if it is a countable union of nowhere
    dense sets.
\end{defn}

For a proof of the Baire Category Theorem (including the version we state
below), we refer the reader to \cite{KechrisBook}*{Theorems 8.1 and 8.4}.

\begin{thm}[Baire Category Theorem]
    If \(S\) is a completely metrisable topological space, then a meagre subset of \(S\) contains no non-empty open subset of \(S\).
    In particular, if \(S\neq \varnothing\), then \(S\) is not a meagre subset of itself.
\end{thm}

\begin{proof}[Proof of \cref{prop:topfact}]
Note that a countable union of meagre sets is meagre.
Thus, by the Baire Category Theorem, there exists \(i\in \N\) such that \(S_i\) is not meagre.
As singleton sets are trivially meagre, it follows that \(S_i\) is uncountable.
It also follows that \(\overline{S_i}\) contains a non-empty open set.
In particular, from \cref{def:topologyonpartials}, there exists \(p\in P_\N\)  and finite \(D\subseteq \N\) such that \(\dom(p)\subseteq D\) and
\(\overline{S_i}\) contains 
\[U \coloneqq \makeset{f\in \operatorname{Bijections}(X,Y)}{\(f|_{D} =p\)}\neq \varnothing.\]
The fact that this set is non-empty implies that \(p\) is a partial bijection from \(X\) to \(Y\). Suppose for a contradiction that there is a partial bijection \(p'\)  from \(X\) to \(Y\) with finite domain extending \(p\) such that there is no \(s\in S_i\) with \(s|_{\dom(p')}=p'\).
It follows that 
\[\makeset{f\in \operatorname{Bijections}(X,Y)}{\(f|_{D\cup \dom(p')} =p'\)}\]
is a non-empty open subset of \(U\) disjoint from \(S_i\). This contradicts the assertion that \(\overline{S_i}\supseteq U\).
\end{proof}

\section{The Symmetric Group}
\label{sec:symgroup}

This section covers our study of the commutative subgroup width of \(\Sym(\N)\) and
\(\Sym(X)\) for arbitrary infinite sets \(X\). We begin with a short proof of the
fact that for a group, the commutative submonoid, subgroup and inverse submonoid
width all coincide. We thus hereafter use the subgroup width.

\begin{lemma}\label{independance_of_definition}
    If \(G\) is a group, then \(\operatorname{cgw}(G) = \operatorname{cmw}(G)
    = \operatorname{ciw}(G)\).
\end{lemma}
\begin{proof}
    First note that inverse submonoids of \(G\) and subgroups of \(G\) coincide,
    and so we need only show that
    \(\operatorname{cgw}(G)=\operatorname{cmw}(G)\). As subgroups are
    submonoids, we trivially have \(\operatorname{cmw}(G)\leq
    \operatorname{cgw}(G)\). If \(\operatorname{cmw}(G)=\infty\), then clearly
    \(\operatorname{cmw}(G)\geq \operatorname{cgw}(G)\). Otherwise let \(k
    =\operatorname{cmw}(G)\), and let \(C_0,C_1,\ldots, C_{k-1}\) be commutative
    submonoids of \(G\) such that \(G=C_0C_1\ldots C_{k-1}\). It follows that
    the same equality holds when we replace each \(C_i\) with the subgroup it
    generates. As the group generated by a set of commuting elements is abelian,
    it follows that \(k\geq \operatorname{cgw}(G)\) as required. \end{proof}

The following lemma is more general than we require in order to study the
commutative subgroup width of infinite symmetric groups. However, we nonetheless
prove it in this level of generality to allow us to apply it in
Section~\ref{sec:monoids}.
\begin{lemma}
    \label{equiv-A-inj}
    Let \(A \leq P_{\N}\) be a commutative submonoid. Write \(A_\textrm{inj}\)
    to denote the submonoid of \(A\) of all fully defined injective functions in
    \(A\). Let \(\sim\) be the least equivalence relation containing the binary
    relation:
    \[
        \{(x, y) \in \N^2 \mid \textrm{ there exists } f \in A_\textrm{inj} \textrm{ with }
        (x)f = y\}.
    \]
    If \(x \sim y\) and \(f, g \in A_\textrm{inj}\), then
    \((x)f = (x)g\) if and only if \((y)f = (y)g\).
\end{lemma}

\begin{proof}
    It suffices to show that if \(x \sim y\) and \(f, g \in A_\textrm{inj}\)
    then \((x)f = (x)g\) implies \((y)f = (y)g\); the converse follows
    by symmetry. So let \(x, y \in \N\) be such that \(x \sim y\) and
    let \(f, g \in A_\textrm{inj}\) be such that \((x)f = (x)g\). We
    first consider the cases where \((x)h = y\) or \((y)h = x\) for
    some \(h \in A_\textrm{inj}\). If \((x)h = y\), then
    \[
        (y)f = (x)hf = (x)fh = (x)gh = (x)hg = (y)g,
    \]
    as required. Similarly, if \((y)h = x\), then we have 
    \[
        (y)fh = (y)hf = (x)f = (x)g = (y)hg = (y)gh. 
    \]
    Since \(h\) is injective, it follows that \((y)f = (y)g\), as
    required. We have thus shown that the result holds when
    \((x)h = y\) or \((y)h = x\) for some \(h \in A_\textrm{inj}\). Since
    \(\sim\) is just the transitive closure of such `elementary moves', it
    follows that \(x \sim y\) is sufficient to reach the conclusion
    that \((y)f = (y)g\).
\end{proof}

We now aim to prove that \(\operatorname{cgw}(\Sym(\N) > 3\). Our proof is by
contradiction, and thus we assume that \(\Sym(\N)\) is the product of three
abelian subgroups: \(A\), \(B\) and \(C\). We begin with a lemma that discusses
\(\sim\)-classes as defined in \cref{equiv-A-inj}. This covers one of the four
cases that arise dependent on whether or not every \(\sim_A\)-class and every
\(\sim_C\)-class is finite, by showing that the case of when neither is true
never occurs. For the proof, recall that a \emph{moiety} is a subset \(M\) of a
set \(X\) such that \(|M| = |X| = |X \setminus M|\).

\begin{lemma}\label{lem:LBSym1}
 Suppose that \(A,B,C\leq \Sym(\N)\) are abelian groups such that \(ABC=\Sym(\N)\).
Let \(\sim_{A}\), \(\sim_B\) and \(\sim_C\) be as in \cref{equiv-A-inj}. Then at least one
of the following holds:
\begin{itemize}
    \item Every \(\sim_A\)-class is finite;
    \item Every \(\sim_C\)-class is finite.
\end{itemize}
\end{lemma}
\begin{proof}
    Suppose for a contradiction that \(x,z\in \N\) are such that \([x]_{\sim_A}\) and \([z]_{\sim_C}\) are infinite.
    Fix a moiety \(M_A\) of \([x]_{\sim_A}\) and a moiety \(M_C\) of \([z]_{\sim_C}\).
    Note that, by \cref{equiv-A-inj}, if \(f \in A'\coloneqq \{a|_{[x]_{\sim_A}} \mid a\in A\}\), then \(f\) is entirely determined by where it sends \(x\), and thus
    \(A'\) is a countable abelian group. By symmetry, \(C' \coloneqq \{c|_{[z]_{\sim_C}} \mid c\in C\}\) is also a countable abelian group.
    
    As \(M_A\) and \(M_C\) are moieties of \(\N\), it follows that every bijection \(f \colon M_A\to M_C\) has an extension \(\bar{f}\in \Sym(\N)\).
    In particular, for all \(f\in \textrm{Bijections}(M_A,M_C)\) there exists \((a_f,b_f,c_f)\in A\times B\times C\) such that
    \[(a_fb_fc_f)|_{M_A} = f.\]
    Let \(a_f' \coloneqq \id_{[x]_{\sim_A}}a_f\in A'\) and \(c_f' \coloneqq c_f \id_{[z]_{\sim_C}}\) and note that
    \(a_f'b_fc_f' = f\) as well (using composition of partial transformations of \(\N\)). In particular, 
    \[\bigcup_{a\in A'}\bigcup_{c\in C'}\makeset{abc}{\(b\in B\) and \((M_A)abc=M_B\)}=\textrm{Bijections}(M_A,M_C).\] 
    By \cref{prop:topfact}, there exist \(a\in A'\) and \(c\in C'\) and a partial bijection \(q\) from \(M_A\) to \(M_C\) such that if 
    \[S \coloneqq \makeset{abc}{\(b\in B\) and \((M_A)abc=M_C\)},\]
    then the following hold:
    \begin{itemize}
        \item \(S\) is uncountable;
        \item \(q\) has finite domain;
        \item if \(q'\) is a partial bijection from \(M_A\) to \(M_C\) with finite domain extending \(q\), then there is an element of \(S\) extending \(q'\).
    \end{itemize}
   Equivalently, if we define \(M_A'\coloneqq M_Aa\) and \(M_C'\coloneqq M_Cc^{-1}\), then there is a partial bijection \(p\) from \(M_A'\) to \(M_C'\) with finite domain such that if \(p'\) is a partial bijection from \(M_A'\) to \(M_C'\) with finite domain extending \(p\), there is an element of \(B\) extending \(p'\).
   Thus for all \(m\in M_A'\setminus \dom(p)\) and \(n\in M_C'\setminus \im(p)\), we have \(m\sim_B n\).
   As \(\sim_B\) is an equivalence relation and \(M_C'\neq \varnothing\), it follows that \(M_A'\) intersects only finitely many classes of \(\sim_B\).
   From \cref{equiv-A-inj}, we now have that the restriction of an element of \(B\) to \(M_{A}'\) is determined by its action on finitely many points. 
   In particular, the set \(\makeset{b|_{M_A'}}{\(b\in B\)}\) is countable.
   This is a contradiction as \(S\subseteq a\makeset{b|_{M_A'}}{\(b\in B\)}c\) and \(S\) was uncountable.
\end{proof}

We now know at least one of \(\sim_A\) and \(\sim_C\) admits only finite equivalence classes. We will leverage a counting argument to obtain contradictions in each of the remaining three cases arising from whether or not \(\sim_A\), \(\sim_C\) or
both of these equivalence relations have only finite equivalence classes.
\begin{lemma}\label{techlem}
    Suppose that \(X,Y\subseteq \N\) are finite and \(\Gamma\) is an abelian subgroup of \(\Sym(\N)\). 
    Then the cardinality of the set \(S_{\Gamma,X,Y}\coloneqq \makeset{\gamma|_X}{\(\gamma \in \Gamma\) and \((X)\gamma \subseteq Y\)}\) is at most \(3^{|Y|}\). 
\end{lemma}
\begin{proof}
If \(|X|>|Y|\) then \(S_{\Gamma,X,Y}\) is empty.
So we can assume that \(|X|\leq |Y|<\infty\).
 We first partition \(Y\) via the restriction of the equivalence relation \(\sim_\Gamma\) (see \cref{equiv-A-inj}) to \(Y\). 
    We ignore any elements of \(Y\) not \(\sim_\Gamma\) related to at least one element of \(X\) and let the remaining classes be \(Y_0,Y_1,\ldots, Y_{k-1}\).
    Each element of \(S_{\Gamma,X, Y}\) maps each element of \(X\) to an element of \(Y\) with the same \(\sim_\Gamma\) class. 
    In particular, if there is an element of \(X\) which is not \(\sim_\Gamma\)-related to any element of \(Y\), then \(|S_{\Gamma,X,Y}|=0\) and we are done.
    We can therefore assume that for all \(x\in X\), there exists \(i_x<k\) such that \(Y_{i_x}\) is the set of elements of \(Y\) in the same \(\sim_\Gamma\) class as \(x\). 
    
    If multiple elements of \(X\) have the same \(\sim_\Gamma\)-class, then by \cref{equiv-A-inj}, the image of one element from each class is sufficient to determine an element of \(S_{\Gamma,X,Y}\) .
    It follows that \(|S_{\Gamma,X,Y}|\leq |Y_0||Y_1|\cdots |Y_{k-1}|\).
    Let \(n\coloneqq |Y|\), \(m\coloneqq  |X|\) and for each \(i\), let \(n_i\coloneqq |Y_i|\).
    Note also that \(k\leq m\).
    We have that
    \[|S_{\Gamma,X,Y}|\leq n_0n_1\cdots n_{k-1}.\]
    Given \(n_0,n_1,\ldots, n_{k-1}\), if for example, \(n_0-n_1\geq 2\), then
    \[(n_0+1)(n_1-1)=n_0n_1+n_0-n_1-1\geq n_0n_1+1.\]
    More generally, if any two of \(n_0,n_1,\ldots, n_{k-1}\) differ by at least \(2\), then we can find a new sequence of positive integers \(n_0',n_1',\ldots, n_{k-1}'\) with the same sum but a larger product.
    We can repeat this process until all the resulting numbers \(n_0^\ast,\ldots, n_{k-1}^\ast \) differ pairwise by at most \(1\).
    As \(n_0^\ast +n_1^\ast +\cdots + n_{k-1}^\ast  = n_0+n_1+\cdots + n_{k-1}\), the pigeonhole principle tells us that we cannot have every \(n_i^\ast \) satisfying \(n_i^\ast > \frac{n}{k}\geq \frac{n_0+n_1+\cdots + n_{k-1}}{k} \). 
    As the \(n_i^\ast\) differ by at most \(1\), it follows that all of them are bounded above by \(\frac{n}{k}+1\leq \frac{2n}{k}\). In particular,  
    \[|S_{\Gamma,X,Y}|\leq n_0n_1\cdots n_{k-1}\leq n_0^\ast n_1^\ast \cdots n_{k-1}^\ast \leq \left(\frac{2n}{k}\right)^k.\]

    We define \(f_n \colon \R_{> 0} \to \R_{> 0} \) by
    \[(k)f_n=\left(\frac{2n}{k}\right)^k.\]
    Since the natural logarithm \(\log\) is an order preserving function from \(\R_{>0}\) to
    \(\R\), we can investigate the maximum values of \(\log((k)f_n)\).
    The derivative of \(\log((k)f_n)=k(\log(2n)-\log(k))\) with respect to \(k\) is \(\log(2n)-\log(k)-(k / k)=\log(2n/k)-1\), which evaluates to zero precisely when \(k=\frac{2n}{e}\). Moreover, note that \(f_n\) is clearly decreasing on the domain \([2n,\infty)\). Thus \(f_n\) attains a maximum value at \(k=\frac{2n}{e}\), is increasing before this value and is decreasing after this value.
    In particular,
    \[|S_{\Gamma,X,Y}|\leq \left(\frac{2n}{k}\right)^k\leq e^{\frac{2n}{e}}\leq e^n\leq 3^n.
    \qedhere\]
\end{proof}

We now use \cref{techlem} to show that the case of precisely one of \(\sim_A\) and
\(\sim_B\) admitting only finite equivalence classes does not occur, which will
naturally leave the case when both do.
\begin{lemma}\label{lem:SymNLB2}
  Suppose that \(A,B,C\leq \Sym(\N)\) are abelian groups such that \(ABC=\Sym(\N)\).
Define \(\sim_{A}\) and \(\sim_C\) using \cref{equiv-A-inj}. Then
every \(\sim_A\)-class and every \(\sim_C\)-class is finite.
\end{lemma}
\begin{proof}
    Suppose for a contradiction that \(\sim_C\) has an infinite equivalence
    class (the case that \(\sim_A\) has one is symmetric). By \cref{lem:LBSym1},
    we must have that every \(\sim_A\)-class is finite. Let \(x\in \N\) be fixed
    with an infinite \(\sim_C\) class. As \(\N/\sim_A\) is an infinite family of
    finite subsets of \(\N\), we can fix a moiety \(M_A\) of \(\N\) which is a
    union of \(\sim_A\)-classes. Additionally, fix a moiety \(M_C\) of
    \([x]_{\sim_C}\).

    By \cref{equiv-A-inj}, \(A'\coloneqq \makeset{a|_{M_A}}{\(a\in A\)}\) and \(C'\coloneqq \{c|_{[x]_{\sim_C}} |\ c\in C\}\) are well-defined abelian groups. Moreover, \cref{equiv-A-inj} implies that \(C'\) is countable, as every element in \(C'\) is uniquely determined by its action on \(x\).
    Since \(M_A\) and \(M_C\) are moieties of \(\N\), every bijection \(f \colon M_A\to M_C\) has an extension \(\bar{f}\in \Sym(\N)\).
    It follows that for all \(f\in \textrm{Bijections}(M_A,M_C)\) there exists \((a_f,b_f,c_f)\in A\times B\times C\) such that
    \[(a_fb_fc_f)|_{M_A} = f.\]
    Let \(a_f'\coloneqq \id_{M_A}a_f\in A'\) and \(c_f'\coloneqq  c_f\id_{[x]_{\sim_C}}\) and note that
    \(a_f'b_fc_f' = f\) as well. In particular, 
    \[\bigcup_{c\in C'}\makeset{abc}{\(a\in A'\), \(b\in B\) and \((M_A)abc=M_B\)}=\textrm{Bijections}(M_A,M_C).\]
    By \cref{prop:topfact}, there is \(c\in C'\) and a partial bijection \(q\) from \(M_A\) to \(M_C\) with finite domain, such that for all partial bijections \(q'\) from \(M_A\) to \(M_C\) with finite domain extending \(q\), there is an element of
    \[\makeset{abc}{\(a\in A'\), \(b\in B\) and \(\im(abc)=M_C\)}\]
    extending \(q'\).
    Equivalently, if we define \(M_C'\coloneqq M_Cc^{-1}\) and 
   \[S\coloneqq \makeset{ab}{\(a\in A'\), \(b\in B\) and \(\im(ab)=M_C'\)},\]
   then there exists a partial bijection \(p\) from \(M_A\) to \(M_C'\) with finite domain, such that for all partial bijections \(p'\) from \(M_A\) to \(M_C'\) with finite domain extending \(p\), there is an element of \(S\) extending \(p'\).
 Let \(M_A'\) be an infinite subset of \(M_A\) which is a union of \(\sim_A\)-classes and is disjoint from \(\dom(p)\). Similarly, let \(M_C''\coloneqq  M_C'\backslash \im(p)\). Let \(A''\coloneqq A'|_{M_A'}\) and note that this is a well-defined abelian group, since
 \(M_A'\) is a union of \(\sim_A\)-classes.
   We now have that there are elements of \(A''B\) extending every partial bijection from \(M_{A}'\) to \(M_{C}''\) with finite domain.
   Let \(X\subseteq M_{A}'\) be a finite union of \(\sim_A\)-classes large enough that 
   \begin{equation}\label{eqn:factorial}3^{2|X|}< |X|!.\end{equation}
   Let \(Y\subseteq M_{C}''\) be such that \(|Y|=|X|\).
   As there are elements of \(A''B\) extending every bijection from \(X\subseteq M_{A}'\) to \(Y\subseteq M_{C}''\), it follows that
   \[|\makeset{(a'b')|_{X}}{\(a'\in A''\), \(b'\in B\), and \((X)a'b'=Y\)}|=|X|!.\]
As \(X\) is a union of \(\sim_A\) classes, it follows that
   \[\makeset{(a'b')|_{X}}{\(a'\in A''\), \(b'\in B\), and \((X)a'b'=Y\)} =S_{A,X,X}S_{B,X,Y},\]
   using the notation from \cref{techlem}.
   Thus, by \cref{techlem}, we have
   \[|X|!\leq 3^{|X|}3^{|Y|}=3^{2|X|},\]
   which contradicts Eqn~\eqref{eqn:factorial}. 
\end{proof}

The remaining case within which we need to find a contradiction is the case when
both \(\sim_A\) and \(\sim_C\) admit only finite equivalence classes. We can use
\cref{techlem}, but we require further `counting tools' on top of this to reach
a contradiction. The following relates sizes of finite sets of
\(\sim_A\)-classes to sizes of finite sets of \(\sim_C\)-classes.

\begin{lemma}\label{lem:SymNLB3}
Suppose that \(A,B,C\leq \Sym(\N)\) are abelian groups such that
\(ABC=\Sym(\N)\). We also define \(\sim_{A}\) and \(\sim_C\) using
\cref{equiv-A-inj}. There is a constant \(k\in \N\) such that if
\(X\subseteq \N\) is a union of finitely many \(\sim_A\)-classes, \(Y\subseteq
\N\) is a union of finitely many \(\sim_C\)-classes and \(k\leq |X|\leq |Y|\),
then \(|Y|\geq 2|X|\).
\end{lemma}
\begin{proof}
   Fix \(k\in \N\) to be large enough such that for all \(n\geq k\) we have \(3^{3n}< \left \lfloor \frac{n}{2}\right \rfloor^{\lfloor \frac{n}{2} \rfloor}\).
   Suppose that \(X\subseteq \N\) is a union of finitely many \(\sim_A\)-classes, \(Y\subseteq \N\) is a union of finitely many \(\sim_C\)-classes and \(k\leq |X|\leq |Y|\).
     Let \(n\coloneqq |Y|\) and \(m\coloneqq |X|\geq k\).
     We need only show that \(n\geq 2m\).
By \cref{techlem}, we have 
    \[|\makeset{a|_X}{\(a\in A\)}| \leq 3^{m},\quad |\makeset{b|_X}{\(b\in B\) and \((X)b\subseteq Y\)}| \leq 3^{n},\quad
    |\makeset{c|_Y}{\(c\in C\)}| \leq 3^n.\]
    As \(ABC=\Sym(\N)\), it follows that every injective map from \(X\) to \(Y\) is the restriction of an element of \(ABC\). As \(X\) is fixed setwise by \(A\) and \(Y\) is fixed setwise by \(C\), it follows that every injective map from \(X\) to \(Y\) is an element of
    \[S_{A,X,X}S_{B,X,Y}S_{C,Y,Y}=\makeset{a|_X}{\(a\in A\)}\makeset{b|_X}{\(b\in B\) and \((X)b\subseteq Y\)}\makeset{c|_Y}{\(c\in C\)}.\]
    Note that  there are \(\frac{n!}{(m-n)!}\) injective functions from \(X\) to \(Y\).
   It follows that
   \[\left \lfloor \frac{n}{2}\right \rfloor^{\lfloor \frac{n}{2} \rfloor}>3^{3n}\geq 3^{m+2n}\geq |S_{A,X,X}||S_{B,X,Y}||S_{C,Y,Y}|\geq \frac{n!}{(n-m)!}=n(n-1)\cdots (n-m+1).\]
   It follows that \(m<\frac{n}{2}\) as required.
\end{proof}

We require one last `counting tool' before we can directly reach our contradiction in
\cref{thm:atleast4}.
\begin{lemma}\label{bigclasses}
Suppose that \(A,B,C\leq \Sym(\N)\) are abelian groups such that \(ABC=\Sym(\N)\).
We also define \(\sim_{A}\) and \(\sim_C\) using \cref{equiv-A-inj}. Then 
\begin{itemize}
    \item every \(\sim_A\) equivalence class is finite and there are only finitely many classes of each finite cardinality;
    \item every \(\sim_C\) equivalence class is finite and there are only finitely many classes of each finite cardinality.
\end{itemize}
\end{lemma}
\begin{proof}
Note that \(ABC=\Sym(\N)\) if and only if \(CBA=(ABC)^{-1}=\Sym(\N)^{-1}=\Sym(\N)\).
Thus if one of these conclusions holds then so does the other.
We show only the second conclusion.

Recall from \cref{lem:SymNLB2} that all the classes of \(\sim_A\) and \(\sim_C\) are finite.
Let \(k\in \N\) be a constant as described in
    \cref{lem:SymNLB3}.
    Suppose for a contradiction that there exists \(q\in \N\) such that \(\sim_C\) has infinitely many classes of cardinality \(q\).
    Let \(X\) be a finite union of \(\sim_A\)-classes such that \(|X|> k+q\). 
    Let \(Y\) be the union of any \(\left \lceil\frac{|X|}{q}\right\rceil\) equivalence classes of \(\sim_C\) each with cardinality \(q\).
    In particular,
    \[2|X|> |X|+q=q\left(\frac{|X|}{q} +1\right)\geq q\left \lceil\frac{|X|}{q}\right\rceil=|Y| \textrm{ and } |Y|=q\left \lceil\frac{|X|}{q}\right\rceil\geq |X|.\]
    This contradicts the definition of \(k\) (see \cref{lem:SymNLB3}).
\end{proof}

We can now complete out proof by contradiction that \(\Sym(\N)\) cannot be the product
of three abelian subgroups.

\begin{prop}\label{thm:atleast4}
    \(\cw(\Sym(\N))\geq 4\).
\end{prop}
\begin{proof}
Suppose for a contradiction that \(A,B,C\leq \Sym(\N)\) are abelian with \(ABC=\Sym(\N)\). Define \(\sim_A\), \(\sim_B\) and \(\sim_C\) using \cref{equiv-A-inj}.
  Note that \cref{bigclasses} implies that both \(\sim_A\) and \(\sim_C\) have only finite equivalence classes, and we can find equivalence classes of each relation with arbitrarily large finite cardinalities.
  However, as \((0)ABC=\N\), we must have that there is an infinite \(\sim_B\)-class.
Let \(x\in \N\) be such that \([x]_{\sim_B}\) is infinite. It follows that we can fix some \(y\in [x]_{\sim_B}\backslash [x]_{\sim_A}\).
Let \(s\in \N\) be such that  \(|[s]_{\sim_C}|\geq |[x]_{\sim_A}|\).
Let \(p \colon [x]_{\sim_A} \to [s]_{\sim_C}\) be a fixed injective map.
Note that if \(abc\in ABC\) extends \(p\), then \(b|_{[x]_{\sim_A}}=(a^{-1}pc^{-1})|_{[x]_{\sim_A}}\) is an injective map from
\([x]_{\sim_A}\) to \([y]_{\sim_C}\).
In particular, the set 
\[\makeset{b|_{[x]_{\sim_A}}}{there exist \(a\in A\) and \(c\in C\) with \(p= (abc)|_{[x]_{\sim_A}}\)}\]
is a finite set of injective maps from \([x]_{\sim_A}\) to \([s]_{\sim_C}\).
If \(b\in B\), then \cref{equiv-A-inj} tells us that \(b|_{[x]_{\sim_B}}\) is determined by \(b|_{[x]_{\sim_A}}\). It follows that the set
\[\makeset{b|_{[x]_{\sim_B}}}{there exist \(a\in A\) and \(c\in C\) with \(p= (abc)|_{[x]_{\sim_A}}\)}\]
is also finite. Thus, the set \[T\coloneqq \makeset{(y)b}{there exist \(a\in A\) and \(c\in C\) with \(p= (abc)|_{[x]_{\sim_A}}\)}\]
is finite. Let \(\tilde{T}\) be the set of elements of \(\N\) which are \(\sim_C\)~related to elements of \(T\). In particular, \(\tilde{T}\) is finite as well.
Let \(t\in \N\setminus (\tilde{T} \cup [s]_{\sim_C})\) be such that \(|[t]_{\sim_C}|\geq |[y]_{\sim_A}|\).
In particular, we can find a permutation \(f\in \Sym(\N)\) which extends \(p\) and such that \(([y]_{\sim_A})f\subseteq [t]_{\sim_C}\).
By assumption, there exists \(abc \in ABC\) such that \(abc=f\).
From the definition of \(T\), we must have \((y)b\in T\) and so \((y)b \notin [t]_{\sim_C}\).
However,
\[(y)b\sim_C (y)bc=(ya^{-1})abc=((y)a^{-1})f\in ([y]_{\sim_A})f\subseteq [t]_{\sim_C}.\]
  This is a contradiction.
\end{proof}

The following lemma is essentially the same as \cite{abelianprod}*{Lemma 3}, however formulated in more generality so that we can use it for different moieties in later arguments.
We include a proof for completeness.

\begin{lemma}\label{full diagonal}
Suppose that \(X\) is an infinite set and \(D\) is a moiety of \(X\).
There are abelian subgroups \(H_D,V_D\leq \operatorname{Sym}(X)\) such that for all \(f\in \operatorname{Sym}(D)\), there exist \(h_f\in H_D\) and \(v_f\in V_D\) with \( (h_fv_f)|_D=f\).
\end{lemma}
\begin{proof}
Let \(A\) be an abelian group with \(|A|=|X|\). 
We can assume without loss of generality that \(X=A\times A\) and \(D=\makeset{(a,a)}{\(a\in A\)}\). Let \[V_D\coloneqq \makeset{v\in \operatorname{Sym}(A\times A)}{for all \(a\in A\) there exists \(v_a\in A\) such that\\ for all \(b\in A\) we have \((a,b)v=(a,b+v_a)\)},\]
\[H_D\coloneqq \makeset{h\in \operatorname{Sym}(A\times A)}{for all \(a\in A\) there exists \(h_a\in A\) such that\\ for all \(b\in A\) we have \((b,a)h=(b+h_a,a)\)}.\]
For all \((a\mapsto v_a)\in A^A\), we have a unique corresponding element of \(V_D\).
From this we can see that \(V_D\) is isomorphic to the abelian group \( (A^A,+)\).
The group \(H_D\) is also isomorphic to \((A^A, +)\) analogously.
Let \(f\in \operatorname{Sym}(D)\) be arbitrary and let \(f'\in \operatorname{Sym}(A)\) be such that \((a,a)f=((a)f',(a)f')\) for all \(a\in A\). 
For all \(a\in A\), let \(h_a\coloneqq (a)f'-a\) and \(v_a\coloneqq a-(a){f'}^{-1}\). We then define \(h\in H_D\) and \(v\in  V_D\) by
\[(b,a)h=(b+h_a,a)\quad \text{and}\quad (a,b)v=(a,b+v_a).\]
In particular, for all \((a,a)\in D\), we have
\begin{align*}
    (a,a)hv&= (a+h_a,a)v\\
    &= (a+(a)f'-a,a)v\\
    &= ((a)f',a)v\\
    &=((a)f',a+v_{(a)f'})\\
    &=((a)f',a+(a)f'-(a)f'{f'}^{-1})\\
    &=((a)f',(a)f')\\
    &=(a,a)f. \qedhere
\end{align*}
\end{proof}

If \(X\) is a set and \(S\subseteq X\), then we write \(\operatorname{Sym}_X(S)\) for the group of permutations of \(X\) which are supported on \(S\); that is,
\(\{f \in \Sym(X) \mid (x)f = x \text{ for all } x \in X \setminus S\}\). For an
infinite set \(X\), it is immediate that \(\operatorname{Sym}_X(S)\cong \operatorname{Sym}(S)\).

\cref{full diagonal} is very powerful as it, in some sense, gives us a copy of
the entire symmetric group using only two abelian groups. However, it has the
disadvantage that it doesn't give a clear idea what is happening to the points
not in \(D\). The following lemma rectifies this issue at the cost of three
additional abelian groups.

\begin{lemma}\label{moietysupp}
Suppose that \(X\) is an infinite set and \(D\) is a moiety of \(X\). Let \(\operatorname{Sym}_X(D)\) be the group of permutations of \(X\) supported on \(D\). Then there exist abelian subgroups \(A_0,A_1,A_2,A_3,A_4\leq \operatorname{Sym}(X)\) such that
\[\operatorname{Sym}_X(D)\subseteq A_0A_1A_2A_3A_4.\]
\end{lemma}
\begin{proof}
    Let \(D'\) be a moiety of \(X\setminus D\).
    By \cref{full diagonal}, there are abelian subgroups \(A_0,A_1\leq \operatorname{Sym}(X)\) such that every permutation of \(D\cup D'\) has an extension in \(A_0A_1\). 
    Note that \[|D\cup D'|=\sum_{i\in \N\cup \{\aleph_0\}} i|X|.\]
    In particular, we can partition \(D\cup D'\) into \(|X|\) copies of each countable cardinal.
    If we define \(A_3\) to act as the cyclic group of order \(i\) on each of these sets, then \(A_2\) is an abelian group and contains elements with every possible disjoint cycle structure which has \(|X|\) fixed points.
    Let \(A_2\leq \operatorname{Sym}(X)\) be abelian, supported on \(D\cup D'\) and such that every disjoint cycle structure with infinitely many fixed points in \(D\cup D'\) is realised by \(A_2\).
   It follows that for all \(g\in \operatorname{Sym}_X(D)\) there exist \(h_g\in A_0\) and  \(v_g\in A_1\) such that \(g\in (h_gv_g)^{-1}A_2h_gv_g\). In particular,
    \[\operatorname{Sym}_X(D)\subseteq A_1A_0A_2A_0A_1,\]
    as required.
\end{proof}

We next introduce the last technical lemma needed to complete the proof of the upper bound from Theorem A. 
This is the only step in the proof which requires any assumptions on the cardinality of \(X\) (other than the assumption of being infinite).
Recall that an infinite cardinal \(\kappa\) is called \emph{regular} if it cannot be expressed as a sum of less than \(\kappa\) cardinals strictly below \(\kappa\). 
For example, \(\aleph_0\) is regular, since a union of finitely many finite sets is finite and every infinite successor cardinal is regular since \(\aleph_\alpha\cdot \aleph_\alpha =\aleph_\alpha<\aleph_{\alpha+1}\) for all ordinals \(\alpha\).
On the other hand, \(\aleph_{\omega}\) is not regular since
\(\aleph_{\omega}=\sum_{i\in \N}\aleph_i.\)
\begin{lemma}\label{lem:transfinite}
    Suppose that \(X\) is an infinite set, \(D\) is a moiety of \(X\), and \(|X|\) is a regular cardinal.
    There are abelian groups \(C,H\leq \Sym(X)\) such that for all moieties
    \(M\) of \(X\), there exist \(c_{M}\in C\) and \(h_{M}\in H\) with
    \[(M)c_{M}h_{M}\cap D=\varnothing.\]
\end{lemma}
\begin{proof}
    Without loss of generality, we can assume that there is an abelian group \(A\) with identity \(0_A\) and a set \(D'\) such that 
    \begin{itemize}
        \item \(|A|=|D'|=|X|\);
        \item \(X=D'\times A\);
        \item \(D=D'\times \{0_A\}\).
    \end{itemize}
    For all \(f\in A^{D'}\), let \(h_f \colon X\to X\) be defined by
    \((d,a)h_f=(d,a+(d)f).\)
    Let \(H\) be the group of all permutations \(h_f\) for  \(f\in A^{D'}\). 
    In particular \((H, \circ) \cong (A^{D'},+)\)  is abelian.
Fix enumerations 
\(\makeset{a_\beta}{\(\beta\) is an ordinal below \(|X|\)}\)
and \(\makeset{d_\beta}{\(\beta\) is an ordinal below \(|X|\)}\) of \(A\) and \(D'\), respectively.
For all \(\beta<|X|\), define
\[B_\beta \coloneqq \makeset{(d_j,a_k)\in X}{\(\max(j, k)=\beta\)}.\]
Note that each of the sets \(B_\beta\) is non-empty and has cardinality at most \(|\beta+1||\beta+1|<|X|\), and also these sets partition \(X\).
On each of the sets \(B_\beta\), we define an abelian group operation \(+_\beta\).
For all \(f=(f_{\beta})_{\beta<|X|}\in \prod_{\beta<|X|} (B_\beta,+_{\beta})\), we define a permutation 
\(c_f \colon X\to X\) by
\((x)c_f=x+_{\beta} f_{\beta}\) whenever \(x\in B_\beta\).
Let \(C\) denote the set of all such maps \(c_f\) and note that \(C\) is isomorphic to the abelian group \(\prod_{\beta<|X|} (B_\beta,+_{\beta})\).
We show that \(C\) and \(H\) have the required property.

Let \(M\) be a moiety of \(X\). We will show that we can find \(c_Mh_M\in CH\) such that \((M)c_Mh_M\cap D=\varnothing\).
Let \(I\) denote the set of ordinals \(\beta<|X|\) such that \(B_\beta\not \subseteq M\). In particular, 
\[|X|=|X\backslash M| \leq |\cup_{\beta\in I} B_\beta|\leq \sum_{\beta\in I} |B_\beta|.\]
As \(|X|\) is a regular cardinal and each \(B_\beta\) has cardinality strictly below \(|X|\), it follows that \(|I|=|X|\).
Note that \(|X|\) and \(I\) are both well-ordered sets of cardinality \(|X|\), such that no proper downwards closed subset has cardinality \(|X|\).
As such, we can build a unique order isomorphism \(\phi \colon |X|\to I\) by transfinite induction. Moreover, \(\phi\) satisfies \((\beta)\phi\geq \beta\) for all \(\beta<|X|\). In particular, \((d_\beta, a_{(\beta)\phi})\in B_{(\beta)\phi}\) for all \(\beta<|X|\).
For all \(\beta<|X|\), we fix a choice of \(c_\beta\in B_{(\beta)\phi}\) such that \(c_\beta \notin M\), which exists since \((\beta)\phi\in I\), and so \(B_{(\beta)\phi}\not \subseteq M\).
We then define \(c_M\in C\) to be the element which fixes all elements of the sets \(B_\beta\) when \(\beta\notin I\) and such that
\[(x)c_M= 
    x+_{(\beta)\phi} (d_\beta, a_{(\beta)\phi}) -_{(\beta)\phi} c_{\beta}, \]
    whenever \(x\in B_{(\beta)\phi}\).
    In particular, for all \(\beta<|X|\), we have \[(c_\beta)c_M=c_{\beta}+_{(\beta)\phi} (d_\beta, a_{(\beta)\phi}) -_{(\beta)\phi} c_{\beta}=(d_\beta,a_{(\beta)\phi}).\]
    Let \(h_M\in H\) be the map \((d_\beta, a_\beta) \mapsto (d_\beta, a_\beta-a_{(\beta)\phi})\). Then, for all \(\beta <|X|\), we have
    \[(c_\beta)c_Mh_M=(d_\beta,a_{(\beta)\phi})h_M=(d_\beta, 0_A)\in D.\]
    As the elements \(c_{\beta}\) for \(\beta<|X|\) do not belong to \(M\), it follows that \(D\subseteq (X\setminus M)c_Mh_M\).
    In other words, \((M)c_Mh_M \cap D=\varnothing\), as required.
\end{proof}

\cref{lem:transfinite} allows us to use two abelian groups to move an arbitrary moiety into a fixed moiety. We can then apply \cref{full diagonal} to use another two abelian groups to  move the result to the fixed moiety in any way we wish, and \cref{moietysupp}
allows us to use five abelian groups to permute the remaining points in any way we wish.

\begin{thm}\label{thm:symupper}
    \(\operatorname{cgw}(\operatorname{Sym}(X))\leq 9\) for all infinite sets \(X\) such that \(|X|\) is a regular cardinal.
\end{thm}
\begin{proof}
Let \(D\) be a moiety of \(X\). By \cref{lem:transfinite} (applied to
\(X\setminus D\)), there are abelian subgroups \(A_0,A_1\leq \Sym(X)\) such that
for all moieties \(M\) of \(X\), there exists \(a_0a_1\in A_0 A_1\)
with \((M)a_0a_1\subseteq D\). Let \(D'\) be a moiety of \(X\setminus D\). By
\cref{full diagonal}, let \(A_2,A_3\leq \Sym(X)\) be abelian subgroups such that
every permutation of \(D\cup D'\) has an extension in \(A_2A_3\). Then 
for all moieties \(M\) of \(X\) and bijections \(\phi\colon M\to D\), there
exists \(a_0a_1a_2a_3\in A_0 A_1 A_2 A_3\)  extending \(\phi\). Equivalently, if \(f\in \Sym(X)\), then
\(fA_0A_1A_2A_3\) contains a permutation fixing \(D\) pointwise. By
\cref{moietysupp}, there are abelian groups \(A_4,A_5,A_6,A_7,A_8\leq \Sym(X)\)
such that every permutation fixing \(D\) pointwise belongs to
\(A_4A_5A_6A_7A_8\). Thus 
\[\Sym(X)\subseteq A_4A_5A_6A_7A_8(A_0A_1A_2A_3)^{-1}=A_4A_5A_6A_7A_8A_3A_2A_1A_0,\]
as required.
\end{proof}

Combining \cref{thm:atleast4} and \cref{thm:symupper} now completes our bounds on \(\operatorname{cgw}(\Sym(\N))\).
\begin{theorem_no_number}[A]
   \(4 \leq \operatorname{cgw}(\Sym(\N)) \leq 9\). 
\end{theorem_no_number}
\begin{cor}
    \(4\leq \cw(\operatorname{Inj}(\N)) \leq 10\)
\end{cor}
\begin{proof}
    Using Theorem A and noting that \(\cw(\Sym(\N)) =
    \operatorname{cgw}(\Sym(\N))\) (\cref{independance_of_definition}), it
    suffices to show that \(\cw(\operatorname{Inj}(\N)) \in \{\cw(\Sym(\N)),
    \cw(\Sym(\N)) + 1\}\). Suppose that \(C_0,C_1,\ldots, C_{k-1}\leq \operatorname{Inj}(\N)\)
    are commutative submonoids  showing that \(\cw(\operatorname{Inj}(\N))=k\).
    Then \(C_0C_1\ldots C_{k-1}\supseteq \Sym(\N)\).
    As \(\Sym(\N)\) is the complement of an ideal of \(\operatorname{Inj}(\N)\), it follows that
    \[(C_0\cap \Sym(\N))(C_1\cap \Sym(\N))\ldots (C_{k-1}\cap \Sym(\N))= \Sym(\N).\]
    Thus \(\cw(\Sym(\N))\leq \cw(\operatorname{Inj}(\N))\).
    
    It remains to show that \(\cw(\operatorname{Inj}(\N))\leq \cw(\Sym(\N))+1\).
    Suppose that \(C_0,C_1,\ldots, C_{k-1}\leq \Sym(\N)\) are commutative subgroups  showing that \(\cw(\Sym(\N))=k\).
    Let \(\makeset{M_i}{\(i\in \N \cup \{\aleph_{0}\}\)}\) be a partition of \(\N\) into moieties.
    For each \(i\in \N \cup \{\aleph_{0}\}\), let \(f_i \colon M_i\to M_i\) be an
    injective function with \(|M_{i}\setminus \im(f_i)|=i\) and let
    \(\bar{f_i}\) be the extension of \(f_i\) fixing all points not in \(M_i\).
    Let \(C_{-1}\) be the monoid generated by the elements \(f_i\) for \(i\in \N \cup \{\aleph_{0}\}\).
    Note that \(C_{-1}\) is commutative and \(\{|\N \setminus \im(c)| \mid c\in  C_{-1}\} = \N \cup \{\aleph_0\}\). 
    It follows that
    \[C_{-1}C_0C_1\ldots C_{k-1}=C_{-1}\Sym(\N)=\operatorname{Inj}(\N). \qedhere\]
\end{proof}
%
%
%

\section{Partial Transformation Monoids}
\label{sec:monoids}
This section covers the proofs of Theorems B and C; that is, that the commutative submonoid
widths of \(\N^\N\), \(P_\N\) and \(I_\N\) are all \(3\), and the commutative inverse
submonoid width of \(I_\N\) is infinite. Recall that \(\N^\N\) is just the monoid
of functions from \(\N\) to itself, and thus \(\N^\N \leq P_\N\). We can do some of
the work for all three cases by working with submonoids of \(P_\N\), since this
contains both \(I_\N\) and \(\N^\N\).

\begin{prop}
\label{PN-lower-bound}
Let \(M\) be a submonoid of \(P_{\N}\) containing \(\Sym(\N)\).
Then
    \(\cw(M) > 2\).
\end{prop}

\begin{proof}
    Suppose for a contradiction that \(M = AB\), for commutative
    submonoids \(A, B \leq M\).
    Let \(\sim\) be the equivalence relation on \(\N\) defined in 
    \cref{equiv-A-inj}, using \(A\) as our commutative submonoid. We consider
    the cases when all \(\sim\)-classes are finite,
    and when there exists \(x \in \N\) such that \([x]_\sim\) is infinite.  

    Case 1: Every \(\sim\)-class is finite. Note that for all \(x\in \N\) and
    \(f\in A_{\textrm{inj}}\), we have \(([x]_{\sim})f \subseteq [x]_{\sim}\). As any
    injective transformation of a finite set is surjective, it follows that all
    elements of \(A_{\textrm{inj}}\) permute each class of \(\sim\). Thus
    \(A_{\textrm{inj}}\) is a commutative submonoid of \(\Sym(\N)\). As each
    element of \(\Sym(\N)\) is the product of an element of
    \(A_{\textrm{inj}}\subseteq \Sym(\N)\) and an element of
    \(B\) (that is, a permutation and a function), it follows
    that 
    \[\Sym(\N)=A_{\textrm{inj}}(B \cap \Sym(\N)).\]
    This contradicts \cref{thm:atleast4}.

    Case 2: There exists \(x \in \N\) such that \([x]_\sim\) is infinite. Set
    \(X = [x]_\sim\), and define
    \[
        A' = \{f|_X \mid f \in A_\textrm{inj}\}.
    \]
    By the definition of \(\sim\), we have that \((X)A_\textrm{inj} \subseteq
    X\), and so the restriction map \(f \mapsto f|_X\) is a homomorphism. Thus
    \(A'\) is a quotient of \(A_\textrm{inj}\), and is therefore a commutative
    monoid. 
    We will show that \(\Sym(X) \subseteq A' B\) and
    that \(A'\) is countable. For the latter claim, by \cref{equiv-A-inj},
    we have that each element \(f \in A'\) is entirely determined by
    \((x)f\). As \((x)f \in X\), which is countable, there are at most
    countably many choices for \((x)f\) and thus \(f\). In particular,
    \(A'\) is countable.

     We next show that \(\Sym(X) \subseteq A' B\). Let \(\Sym_\N(X) = \{f \in
     \Sym(\N) \mid (y)f = y \textrm{ for all } y \in  \N \setminus X \} \leq
     \Sym(\N)\). Since \(AB = M\), we have that \(\Sym(\N) \subseteq A_\text{inj}
     B\). It follows that \(\Sym_\N(X) \subseteq A_\textrm{inj} B\). Let
     \(h \in \Sym(X)\) be arbitrary. Let \(\bar{h} \in \Sym_\N(X)\) be such that
     \(h = \bar{h} |_X\). Then there exist \(f \in A_\textrm{inj}\) and
     \(g \in B\), such that \(\bar{h} = fg\). Thus
     \[
        f|_X g = (fg)|_X = \bar{h}|_X = h,
     \]
     and so \(h \in A' B\), as required.
     
    We have now established that \(\Sym(X) \subseteq A' B\), and so
    \[\Sym(X)=\bigcup_{a\in A'}aB \cap \Sym(X).\]  
    By \cref{prop:topfact},
 there exists \(a'\in
    A'\) and a partial permutation \(p\) of \(X\) with finite domain such
    that  
    for every finite partial permutation \(p'\) extending \(p\), we have that
    \(a'B\) contains an extension of \(p'\).
    So for all finite partial permutations
    \(p'\) of \(X\) extending \(p\), there exists \(b\in B\) with
    \((a'b)|_{\dom(p')}= p' \). Since for each \(y \in X \setminus
    \dom(p)\) and \(z \in X \setminus \im(p)\) there exists a partial
    function \(p'\) with finite domain such that \(p'|_{\dom(p)} = p\) and \((y)p' = z\), we have that if \(y\in
    X\setminus \dom(p)\), then \((y)a'B \supseteq X\setminus
    \im(p)\). Let \(y \in X\setminus \dom(p)\) be fixed. For all  \(z\in X\setminus \im(p)\) fix some \(b_z\in
    B\) with
    \((y)a'b_z=z\).
    For all \(b\in B\), it follows that 
    \[
        (\{z\})b
        = (\{(y)a'b_z\})b 
        = (\{(y)a'\})b_zb 
        = (\{(y)a'\})bb_z 
        = (\{(y)a'\}b|_{\im(p)\cup \{(y)a'\}})b_z.
    \]
    Thus 
    if \(b\in B\), then \(b|_X\) is determined by
    \(b|_{\im(p)\cup \{(y)a'\}}\). It follows
    that \(B|_{X}\) is countable. As \(A'\leq \operatorname{Inj}(X)\), it follows
    that \(|A'B|\leq |A'||B|_X|\). Thus
    \(A'B\) is countable. This is a contradiction, as
    \(A'B \supseteq \operatorname{Sym}(X)\).
\end{proof}

Unfortunately, unlike our proof that the commutative submonoid width of the three
monoids we are studying is at least \(3\), the remaining
direction requires two differing arguments. The proofs are `morally' similar, but cannot
be applied to \(\Sym(\N)\); firstly \(\cw(\Sym(\N)) > 3\) by \cref{thm:atleast4}, and
additionally, the arguments in the proofs strongly leverage some form of non-injectivity. We begin with the full transformation monoid.
\begin{prop}
    \label{X^X-upper}
    If \(X\) is an infinite set, then \(\cw(X^X)\leq 3\).
\end{prop}
\begin{proof}


We show that there exist commutative submonoids \(A, B, C \leq X^X\) such that
    \(X^X = ABC\). 
    Define an operation \(+\) on \(X\)
    such that \((X,+)\) is a commutative monoid with identity \(0_X\).
    
    Let \(\phi \colon  X \to X^2\) be a bijection and
    define \(\Psi \colon (X^X,+) \to ((X^2)^{(X^2)},\circ)\) by
    \[(n,m)(f)\Psi=(n,m+(n)f).\]
    We must verify that \(\Psi\) is a homomorphism.
    For \(f, g\in X^X\) and \((n,m)\in X^2\), we have
    \[
        (n,m)(f+g)\Psi =
        (n,m+(n)(f+g)) = (n,m+(n)f+(n)g) = (n,m+(n)g)(f)\Psi = (n,m)(f)\Psi(g)\Psi.
    \]
    Thus \(\im(\Psi)\) is a commutative monoid. Let \(S \coloneqq \makeset{\phi
    \circ f \circ \phi^{-1}}{\(f\in \im(\Psi)\)}\) and note that \(S\) is a
    commutative monoid, since \(\im(\Psi)\) is. Let \(l \colon X\to X^2\) be the map
    \(x\mapsto (x,0)\) and \(r \colon X^2 \to X\) be the map \((a,b)\mapsto
    b\). For all \(f\in X^X\) and \(n\in X\), we have
    \begin{align*}
        (n)l\circ (f)\Psi \circ  r= (n,0)(f)\Psi \circ  r=(n,(n)f)r=(n)f.
    \end{align*}
    Thus \(l\circ (f)\Psi \circ  r =f\) for all \(f\in X^X\).
    Equivalently, \((l\phi^{-1}) (\phi(f)\Psi \phi^{-1}) (\phi r) =f\) for all \(f\in X^X\). This implies that
    \[\langle l\phi^{-1} \rangle\cdot S\cdot \langle \phi r \rangle=X^X,\]
    We have expressed \(X^X\) as a product of three commutative submonoids as required.
%
%
%
\end{proof}

In the cases of \(I_X\) and \(P_X\), we can leverage the undefined points as a means
of non-injectivity to produce a slightly easier proof than that for \(T_\N\).
\begin{prop}\label{IX-upper}
Suppose that \(X\) is an infinite set and \(I_X\leq M\leq P_X\).
Then
\(\operatorname{cmw}(M)\leq 3\).
\end{prop}

\begin{proof}
    We show that there exist
    commutative submonoids \(A, B, C \leq M\) such that
    \(M = ABC\). 
    Let \(t,u\colon X\to X\) be injective functions with disjoint images.
    Let \(S=t^{-1}Mu^{-1}\).
    Note that for all \(f,g\in S\), we have \(\im(f)\cap \dom(g)\subseteq \im(t)\cap \im(u)=\varnothing\).
    Thus the product of any two elements of \(S\) is the empty function (which also belongs to \(S\)).
    Hence \(S\) is a zero semigroup and in particular, \(S\) is commutative.
    Let \(B\coloneqq S\cup \{\operatorname{id}_X\}\) and note that \(B\) is a commutative submonoid of \(M\).
    Define \(A\) and \(C\) to be the submonoids generated
    by \(t\) and \(u\), respectively.
    We need only show that \(ABC=S\).
    For \(f\in S\), we have
    \[f=(tt^{-1})f(u^{-1}u)=t(t^{-1}fu^{-1})u\in tSu\subseteq ABC,\]
    as required.
\end{proof}

Combining \cref{PN-lower-bound}, \cref{X^X-upper} and \cref{IX-upper} now gives the following.
\begin{theorem_no_number}[B]
    \(\operatorname{cmw}(\N^\N) = \operatorname{cmw}(P_\N) = \operatorname{cmw}(I_\N) = 3\).
\end{theorem_no_number}

We conclude with a theorem that is stronger than Theorem C; that is, we prove
that for every infinite set, the commutative inverse submonoid width of \(I_X\) is
infinite.

\begin{thm}
If \(X\) is an infinite set, then
\(\operatorname{ciw}(I_X)=\infty\).
\end{thm}
\begin{proof}
Suppose for a contradiction that \(\operatorname{ciw}(I_X)=k<\infty\).
Then there exist commutative inverse submonoids \(C_0,C_1,\ldots, C_{k-1}\leq I_X\) with
\[C_0C_1 \cdots C_{k-1}=I_X.\]
Let \(f \in I_X\) be an everywhere-defined injective function which is
not surjective.  Then \(f=c_0c_1\ldots c_{k-1}\), where \(c_i\in C_i\) for \(i<k\).
As \(c_i\in C_i\), it follows that \(c_ic_i^{-1}=c_i^{-1}c_i\). In particular,
\(\dom(c_i)=\im(c_i)\) for all \(i<k\). As \(f\not\in \operatorname{Sym}(X)\),
there must be a smallest \(j<k\) such that \(\dom(c_j)\neq X\). In particular,
there exists \(x\in X\setminus\dom(c_j)\). It follows that
\[x \notin \operatorname{dom}(c_j\cdots c_{k-1})= \operatorname{dom}((c_0c_1\cdots c_{j-1})^{-1}f).\]
As \((c_0c_1\cdots c_{j-1})^{-1}\) is a permutation and \(\dom(f)=X\), this is a contradiction.
\end{proof}
\section{Further questions}

We have narrowed the possibilities for \(\cw(\Sym(\N))\) to the values \(4\), \(5\), \(6\), \(7\), \(8\), and \(9\). As such it is plausible that an exact value can be found.

\begin{question}
    What is the commutative subgroup width of \(\Sym(\N)\)?
\end{question}

Many results for \(\Sym(\N)\) and \(\N^\N\) have been generalised to the
automorphism groups and endomorphism monoids of important structures such as
Fra\"{\i}ss\'{e} limits \cite{macpherson,KechrisRosendal,CameronOligomorphic,MarimonPinskerendomorphismmonoidsandpolymorphismclones}. It is not clear to what extent the results of this
paper generalise to these settings. Perhaps the two most common such structures are the random graph \(R\) and the dense linear order \((\Q,\leq)\). 
This leads us to the following question.
\begin{question}
    Do any of the monoids \(\operatorname{Aut}(\Q,\leq)\), \(\operatorname{Aut}(R)\), \(\operatorname{End}(\Q,\leq)\), \(\operatorname{End}(R)\) have finite commutative submonoid width?
\end{question}

\vspace{-4mm}
\section*{Acknowledgements}
The first named author was supported by the Heilbronn Institute for Mathematical
Research. The second named author was supported by the EPSRC Fellowship grant
EP/V032003/1 ‘Algorithmic, topological and geometric aspects of infinite groups,
monoids and inverse semigroups’. 

\bibliography{biblio}{}
\bibliographystyle{amsplain}

\end{document}